\documentclass[11pt,a4paper,reqno]{amsart}
\usepackage{graphics}
\usepackage{epsfig}
\usepackage{yfonts}
\usepackage{graphicx}
\usepackage[matrix,arrow,tips,curve,ps]{xy}
\usepackage{enumerate, amssymb}
\usepackage{hyperref}
\usepackage{comment}
\usepackage{paralist}

\def\bdi{\begin{diagram}}
\def\edi{\end{diagram}}

\newtheorem{thm}{Theorem}[section]

\newtheorem{cor}[thm]{Corollary}
\newtheorem{lem}[thm]{Lemma}
\newtheorem{prop}[thm]{Proposition}

\theoremstyle{definition}
\newtheorem{defi}[thm]{Definition}
\newtheorem{defis}[thm]{Definitions}
\newtheorem{conj}[thm]{Conjecture}

\newtheorem{conv}[thm]{Convention}
\newtheorem{nota}[thm]{Notation}
\newtheorem{rem}[thm]{Remark}
\newtheorem{rems}[thm]{Remarks}
\newtheorem{exa}[thm]{Example}
\newtheorem{exas}[thm]{Examples}
\newtheorem{desc}[thm]{Description}

\newcommand{\rien}[1]{}

\newcommand{\Sing}{ \operatorname{{\rm Sing}}}

\newcommand{\Aut}{ \operatorname{{\rm Aut}}}

\newcommand{\Oc}{\mathcal{O}}
\newcommand{\Lcal}{\mathcal{L}}

\newcommand{\proj}{\ensuremath{\mathbb{P}}}

\newcommand{\cF}{{\ensuremath{\mathcal{F}}}}

\newcommand{\cE}{{\ensuremath{\mathcal{E}}}}

\newcommand{\cO}{{\ensuremath{\mathcal{O}}}}

\def\NN{{\mathbb N}}
\def\ZZ{{\mathbb Z}}

\def\PP{{\mathbb P}}

\renewcommand{\epsilon}{\varepsilon}
\renewcommand{\phi}{\varphi}

\newcommand{\bnum}{\begin{enumerate}}
\newcommand{\enum}{\end{enumerate}}
\renewcommand{\emptyset}{\varnothing}
\newcommand{\brem}{\begin{rem}}
\newcommand{\brems}{\begin{rems}}
\newcommand{\erem}{\end{rem}}
\newcommand{\erems}{\end{rems}}
\newcommand{\bexa}{\begin{exa}}
\newcommand{\bexas}{\begin{exas}}
\newcommand{\eexa}{\end{exa}}
\newcommand{\eexas}{\end{exas}}
\newcommand{\bdefi}{\begin{defi}}
\newcommand{\edefi}{\end{defi}}
\newcommand{\bdefis}{\begin{defis}}
\newcommand{\edefis}{\end{defis}}
\newcommand{\bcor}{\begin{cor}}
\newcommand{\ecor}{\end{cor}}
\newcommand{\blem}{\begin{lem}}
\newcommand{\elem}{\end{lem}}
\newcommand{\bconv}{\begin{conv}}
\newcommand{\econv}{\end{conv}}
\newcommand{\bconj}{\begin{conj}}
\newcommand{\econj}{\end{conj}}
\newcommand{\bprop}{\begin{prop}}
\newcommand{\eprop}{\end{prop}}
\newcommand{\bthm}{\begin{thm}}
\newcommand{\ethm}{\end{thm}}
\newcommand{\bnota}{\begin{nota}}
\newcommand{\enota}{\end{nota}}
\newcommand{\bsit}{\begin{sit}}
\newcommand{\esit}{\end{sit}}
\newcommand{\be}{\begin{eqnarray}}
\newcommand{\ee}{\end{eqnarray}}
\newcommand{\bproof}{\begin{proof}}
\newcommand{\eproof}{\end{proof}}
\def\ba{\begin{array}}
\def\ea{\end{array}}

\title[
Genera
of curves
on a very general surface in $\PP^ 3$]{Genera of curves on a very
general surface in $\PP^ 3$}

\author{C.\ Ciliberto, F. \ Flamini, M.\ Zaidenberg}
\address{Dipartimento di Matematica, Universit\`a degli
Studi di Roma ``Tor Vergata'', Via della Ricerca Scientifica,
00133 Roma, Italy} \email{cilibert@mat.uniroma2.it,
flamini@mat.uniroma2.it}

\address{Universit\'e Grenoble I, Institut Fourier, UMR 5582
CNRS-UJF, BP 74, 38402 Saint Martin d'H\`eres c\'edex, France}
\email{mikhail.zaidenberg@ujf-grenoble.fr}

\thanks{{\bf Acknowledgements:} The first and second authors have been supported by the
Italian MIUR Project protocol 2010S47ARA\_005 and by GNSAGA of
INdAM. The third author was supported by the French-Italian cooperation project GRIFGA. The authors thank all Institutions which helped them in this collaboration, including their own Departments.
}

\thanks{ \mbox{\hspace{11pt}}{\it 2010 Mathematics Subject Classification}:
14N25, 14J70, 32J25, 32Q45.\\ \mbox{\hspace{11pt}}{\it Key words}:
projective hypersurfaces, geometric genus, algebraic
hyperbolicity}

\date{}

\begin{document}

\begin{abstract} In this paper we consider the question of determining
the geometric genera of irreducible curves lying on a very general
surface $S$ of degree $d\geqslant 5$ in  $\PP^ 3$ (the cases
$d\leqslant 4$ are well known). For all $d\geqslant 4$ we
introduce the set ${\rm Gaps}(d)$ of all non--negative integers
which are not realized as geometric genera of irreducible curves
on a very general surface of degree $d$ in $\PP^ 3$. We prove
that  ${\rm Gaps}(d)$ is finite and, in particular, that ${\rm
Gaps}(5)= \{0,1,2\}$. The set ${\rm Gaps}(d)$ is the union of
finitely many disjoint and separated integer intervals. The first
of them, according to a theorem of Xu, is ${\rm
Gaps}_0(d):=\left[0, \; \frac{d(d-3)}{2} - 3\right]$. We show that
the next one is ${\rm Gaps}_1(d):=\left[\frac{d^2-3d+4}{2}, \; d^2
- 2d - 9\right]$ for all $d\geqslant 6$.  \end{abstract}

\maketitle

\tableofcontents

\section*{Introduction} In this paper we consider the following question:
what are the geometric genera of irreducible curves lying on a
\emph{sufficiently general} surface $S$ of degree $d$ in  $\PP^
3$?

The answer is trivial for  $d\leqslant 3$: in this case $S$ is
rational and
carries curves of any genera
(see Proposition \ref {prop:gaps3} for a more precise result).

For $d\geqslant 4$, the Noether--Lefschetz theorem says that if $S$ is a \emph{very general} surface of degree $d$ in $\PP^ 3$, then all curves on $S$ are complete intersections with 
another surface in $\PP^ 3$ (see \S\ref {sec:not} below).  So, in investigating our  question, we will suppose $S$ very general in the Noether--Lefschetz sense.

It is  well known that on a very general quartic surface in $\PP^ 3$  there are curves of all genera (see Corollary \ref {cor:g4} below). Thus our question starts becoming interesting only for $d\geqslant 5$.

In \S \ref {sec:not} we introduce, for all $d\geqslant 4$, the set
${\rm Gaps}(d)$, i.e., the set of all non--negative integers which
are not realized as geometric genera of  irreducible curves on a
very general surface of degree $d$ in $\PP^ 3$.

By a theorem of Xu (see \cite[Thm.\,1]{Xu1}), for any $d\geqslant
5$ the integer interval ${\rm Gaps}_0(d):=\left[0, \;
\frac{d(d-3)}{2} - 3\right]$ is contained in ${\rm Gaps}(d)$.
By contrast,  the length 4  interval $J_0(d):=\left[
\frac{d(d-3)}{2} - 2,  \;\frac{d(d-3)}{2} +1 \right]$
 has empty intersection with ${\rm Gaps}(d)$:
indeed, it consists of genera of plane sections of $S$, which can
have at most 3 nodes if $S$ is general (cf.\ Proposition \ref
{acz2}).
 Similarly, the length 10 interval $J_1(d):=[d^2 - 2d - 8, d^2 - 2d+1]$
has empty intersection with ${\rm Gaps}(d)$: it consists of genera
of quadric sections of $S$, which can have at most 9 nodes if $S$
is general (cf.\ again Proposition \ref {acz2}).

In \S \ref {sec:nogaps} we prove  that  ${\rm Gaps}(d)$ is finite
and, in particular, that ${\rm Gaps}(5)= {\rm Gaps}_0(5)$ (see
Theorem \ref {prop:nogaps} and Corollary \ref {cor:5}). We do not
exhibit the minimum  $G_d$ such that
 ${\rm Gaps}(d)\subseteq  [0,G_d]$; so finding $G_d$ remains an
open problem. However,  we provide in  Remark \ref {rem:est} an
asymptotic bound for $G_d$.

 The proof of the finiteness of ${\rm Gaps}(d)$ relies
on a result by Chiantini and the first author (see  \cite[Thm.
3.1]{CC})
that extends
 the above
discussion on   the intervals $J_0(d)$ and $J_1(d)$. Let $g_{d,n}$
be the arithmetic genus of complete intersections of a surface $S$
of degree $d$ with a surface of degree $n$, and let $\ell_{d,n}$
be the dimension of the linear system of these complete
intersections on $S$. Then \cite[Thm. 3.1]{CC}  asserts that for
all non--negative integers, the interval $J_n(d):=[g_{d,n}-
\ell_{d,n},g_{d,n}]$ is covered by genera of complete
intersections of $S$ with surfaces of degree $n$ with $\delta\in
[0,\ell_{d,n}]$ nodes, lying in reduced components of the Severi
variety of nodal curves on $S$. The finiteness of ${\rm Gaps}(d)$
follows from the fact that the intervals $J_n(d)$ overlap as soon
as $n\geqslant d$. In fact this argument proves more, since the
curves we find are nodal and lie in reduced components of the
Severi variety.

The set ${\rm Gaps}(d)$ is the union of finitely many disjoint and
separated integer intervals, as in \eqref {eq:totalgaps}.
Determining all of them is a quite tricky and widely open problem.
The proof of the finiteness of ${\rm Gaps}(d)$ might suggest that
these intervals could contain the integer intervals
$I_{n}(d):=[g_{d,n}+1, g_{d,n+1}-\ell_{d,n+1}-1]$, whenever
$g_{d,n}\leqslant g_{d,n+1}-\ell_{d,n+1}-2$, which, as we know
already,
happens only for finitely many $n\geqslant 1$. This is not  true
in general as shown in  \cite [Examples 1.1 and 1.2]{CC}. However,
we prove that this is the case for $d\geqslant 6$ and $n=1$ (see
Theorem \ref {thm:1gaps}). Namely,
 we show that $I_1(d)=\left[\frac{d^2-3d+4}{2}, \; d^2 - 2d - 9\right]$
is the gap interval ${\rm Gaps}_1(d)$ next to ${\rm Gaps}_0(d)$.
The proof is not difficult for $d\geqslant 9$ (see Proposition
\ref {prop:gaps}). It is based on a result by
Clemens-Xu-Chiantini-Lopez
(see Theorem \ref {thm:cl}),
 which bounds from below the geometric genus of a complete intersection
of a general surface $S$ of degree  $d$ in $\PP^ 3$ with a surface
of degree $n$. When $d\geqslant 9$, this bound
forces a curve with geometric genus $g\leqslant d^2 - 2d - 9$ on
$S$ to lie on a surface of degree $n\leqslant 2$, and the
aforementioned  Proposition \ref {acz2} forces $g$ to lie in
$J_0(d)\cup J_1(d)$, which is exactly complementary to $I_1(d)$.
This argument falls short for $6\leqslant d\leqslant 8$, which
requires a more delicate analysis performed in \S \ref
{sec:conjecture}. A reduction step
 in \S \ref {ssec:red}  reduces these cases to verify non--existence of certain
curves on irreducible, but eventually singular surfaces
 of degree $n=3,4$. This requires  in turn
a quite subtle case by case analysis which relies on the
classification of irreducible cubics and quartics in  $\PP^ 3$
(see ~\cite{Bru, BruWal,IsNa,Um,Um2,Ur}).

 The present paper leaves several open problems.
The main one, which we mentioned already, would be to have a
better comprehension of ${\rm Gaps}(d)$, of its subdivision \eqref
{eq:totalgaps} into
disjoint intervals (how many are there?),  and of the constant
$G_d$ introduced above.

 Of course, one might ask similar (and more difficult) questions
for general hypersurfaces in higher dimensional projective space,
 and for complete intersections. Concerning this, it is
worthwhile mentioning the result of Chiantini--Lopez--Ran in \cite
{CLR},
 which implies  that the minimal geometric genus of a subvariety on a
very general hypersurface of degree $d$ in $\PP^N$ goes to
infinity with $d$, for any given $N\ge 2$.

\subsection* {Notation and conventions} We work over the field of complex numbers.
For notation and terminology we refer to \cite{Hart}.
In particular, for $X$ a reduced, irreducible, projective variety,
we denote by $\omega_X$ its dualizing sheaf: when $X$ is
Gorenstein, $\omega_X$ is invertible. For divisors on a smooth
variety $X$, we use the symbols  $\sim$ and $\equiv$ to denote
linear and numerical equivalence, respectively.  We will sometimes
abuse notation and use the same symbol to denote a divisor $D$ on
$X$ and its class in ${\rm Pic}(X)$. Thus $K_X$ will denote a
canonical divisor or the canonical sheaf $\omega_X$.

Recall that an isolated singular point of a surface $S$ is called
a \emph{Du Val} or \emph {irrelevant} or \emph {simple}
singularity if its fundamental cycle in a minimal
desingularization has dual graph of type $A_n, D_n, E_6, E_7,
E_8$.

\section{Preliminaries}\label{sec:not} Let $d$ be a positive integer.
For $\mathcal L_d: =\vert \mathcal O_{\PP^ 3}(d)\vert$ we set
\be\label{400} N_d: =\dim\,(\mathcal L_d)={{d+3}\choose 3}-1\,.
\ee We will denote by $U_d$ the dense open subset of $\mathcal
L_d$ whose points correspond to smooth surfaces.

Recall that, by Noether--Lefschetz theorem (see, e.g.,  \cite
{GH}), the Picard group of a {very general} surface  $X\in U_d$
with $d\geqslant 4$ is generated by $\mathcal O_X(1)$. \emph{Very
general} means that the property holds off the union $\mathcal
N_d$ of countably many proper Zariski closed subsets of $U_d$. The
set  $\mathcal N_d$ is called the \emph{Noether--Lefschetz locus}
in degree $d$.

Given   $X\in U_d$ and a non--negative integer $n$, we let
$\mathcal L_{X,n}:=\vert \mathcal O_X(n)\vert$, and we denote by
$\ell_{d,n}$ its dimension. One has \be\label{4000} \ell_{d,n}=
\Bigg \{  \begin{array}{ccc}
&N_n= \frac {n (n^ 2+6n+11)}6 \,\,\, &\text {if}\,\,\, n<d \\
&N_n-N_{n-d}-1= \frac {d\big ( 3n^ 2-3n (d-4)+(d^
2-6d+11)\big)}{6} -1
 \,\,\, &\text {if}\,\,\, n\geqslant d, \\
\end{array}
\ee and \be\label{4000b} g_{d,n}=\frac {dn(d+n-4)}2+1\, \ee is the
arithmetic genus of the curves in $\mathcal L_{X,n}$.

For an irreducible $X\in \mathcal L_d$ and a non--negative integer
$g$, $\mathcal V_{n,g}(X)$
will denote the locally closed subset of  $\mathcal L_{X,n}$
formed by irreducible curves  of geometric genus $g$. If the
general member of a component of $\mathcal V_{n,g} (X)$ is
\emph{nodal} with $\delta$ nodes, then  $g_{d,n}=g+\delta$.

\begin{defi}\label{def:gaps} Consider the Zariski closure
$V_{n,d,g}$ in $U_d$ of the locus of all $X\in U_d-\mathcal N_d$
such that $\mathcal V_{n,g}(X)\neq \emptyset$.  Let
$V_{d,g}=\cup_{n\in \mathbb N} V_{n,d,g}$. A non-negative integer
$g$ is said to be a $d$--\emph{gap} if $V_{d,g}\neq U_d$. Roughly
speaking, $g$ is a $d$--gap if and only if,  for a very general
surface $X \in \Lcal_d$, one has $\mathcal V_{n,g}(X) = \emptyset$
 for all $n \geqslant 1$. We will denote by ${\rm Gaps}(d)$ the
set of $d$-gaps. A non--negative integer $g\not \in {\rm Gaps}(d)$
will be called a $d$--\emph{non--gap}.
\end{defi}

In studying ${\rm Gaps}(d)$ we may and will assume $d\geqslant 4$,
since:

\begin{prop}\label{prop:gaps3}
${\rm Gaps}(d)=\emptyset$ for $1\leqslant d\leqslant 3$.
\end{prop}

\begin{proof} This is well known in the plane case
(see  \cite[Thm. (1.49)]{HM}). For quadrics and cubics the
proof is the same. Indeed, $\mathcal V_{n,0}(X)$ is nonempty for
$X\in U_d$, with $d\leqslant 3$ (see, e.g., \cite {testa}). Then a
well known deformation argument shows that $\mathcal V_{n,g}(X)$
is nonempty  for all $g\leqslant g_{d,n}$ (see, e.g., \cite
{nobile}).
\end{proof}

We  will abuse notation and, for integers $k,l$ with $k\leqslant
l$, we will write $[k,l]$ for the \emph{integer interval}
$[k,l]\cap\ZZ$, and we call it simply \emph{interval}.

By \cite[Thm.\,1]{Xu1}, every  $g \leqslant g_{d,1} - \ell_{d,1} -
1 = g_{d,1}-4 = \frac{d(d-3)}{2} - 3$ is a $d$--gap, i.e.
\begin{equation}\label{eq:initialgaps}
{\rm Gaps}_0 (d) : = \left[0, \; \frac{d(d-3)}{2} - 3\right]
\subseteq {\rm Gaps} (d).
\end{equation}
$ {\rm Gaps}_0 (d)$ will be called the {\em initial gap interval}
(see Remark \ref {rem:nongap} below).

\section{The range with no gaps}\label{sec:nogaps} In this section we will
show the finiteness of ${\rm Gaps}(d)$ for any $d \geqslant 4$.
To do this, we first recall:

\begin{thm}\label{thm:cc} {\rm (\cite[Thm. 3.1]{CC})}  Let $X\in \Lcal_d$
be a general surface of degree $d \geqslant 4$ in $\proj^ 3$. For
all integers $n \geqslant 1$ and $g\in J_n(d):=[ g_{d,n}-
\ell_{d,n}, g_{d,n}]$ there is  a reduced, irreducible component
$\mathcal V$ of  $\mathcal V_{n,g}(X)$ whose general element is
nodal with $\delta=g_{d,n}-g$ nodes and $\dim(\mathcal
V)=\ell_{d,n}-\delta$. In  particular $J_n(d)\cap {\rm Gaps}
(d)=\emptyset$.
\end{thm}

 Theorem \ref{thm:cc} is stated in \cite {CC}  under the
assumption $n\geqslant d$, but the same (actually easier) argument
works  for $n<d$.

The following consequence is well known:

\begin{cor}\label{cor:g4} One has ${\rm Gaps}(4)=\emptyset$.
\end{cor}
\begin{proof} This  follows from Theorem \ref {thm:cc}, since $g_{4,n}= \ell_{4,n}$
for all positive integers $n$.
\end{proof}

\begin{rem}\label{rem:nongap}
Since $\ell_{d,1} =3$, from Theorem \ref{thm:cc}   one has
$\mathcal V_{1,g} (X) \neq \emptyset$ for  $g\in J_1(d)=[g_{d,1}-3, \;
g_{d,1}]$, i.e., any such $g$ is a $d$--non--gap.
\end{rem}

As a consequence of Theorem \ref {thm:cc}, we obtain:

\begin{thm}\label{prop:nogaps} For integers $n\geqslant d\geqslant 4$, and for
any $g\geqslant g_{d,n-1} -\ell_{d,n-1}$, there is an irreducible
nodal curve of geometric genus $g$ on the general surface of
degree $d$ in $\proj ^ 3$. In particular, one has  ${\rm Gaps}(d)
\subset \left[0,\; g_{d,d-1}-\ell_{d,d-1}-1 \right]= \left[0,\;
\frac{d(d-1)(5d-19)}{6} -1\right]$.
\end{thm}

\begin{proof} By Theorem \ref {thm:cc},
we need to show that for $n\geqslant d$, the union $J_{n-1}(d)\cup
J_n(d)$ is an integer interval, i.e.,
\begin{equation} \label{eq:1}
\ell_{d,n}\geqslant g_{d,n}-g_{d,n-1}-1= \frac {d(2n+d-5)}2-1.
\end{equation} By  \eqref{4000},  \eqref {eq:1} reads
\begin{equation*}\label{eq:2}
3n(n-d+2)+(d^ 2-9d+26)\geqslant 0,
\end{equation*}  which holds for $n\geqslant d$.   \end{proof}

\begin{rem}\label{rem:est} It is possible to give a better estimate
for the minimal integer $G_d$ such that
${\rm Gaps}(d) \subset \left[0,\; G_d \right]$. Indeed, if  $n<d$,
then \eqref {eq:1} reads
\begin{equation}\label{eq:3}
n^ 3+6n^ 2+n(11-6d)-3(d^ 2-5d-2)\geqslant 0.
\end{equation}
For this to hold, it suffices that
\[
n^ 3+6n^ 2-6nd-3d^ 2\geqslant 0 \,\,\, {\rm i.e.}\,\,\, d\leqslant
\sqrt  {\frac  {n^ 3+9n^ 2} 3}-n.
\]
 The latter inequality is fulfilled, e.g., if $d>n\geqslant
\sqrt[3]{12d^ 2}$. For any such $n$, we have $G_d\le
g_{d,n-1}-\ell_{d,n-1}-1$. 
\end{rem}

\begin{cor}\label{cor:5} For any integer $g\geqslant 3$ there is an
irreducible nodal curve of geometric genus $g$
on a general surface of degree $5$ in $\proj ^ 3$, i.e.,  ${\rm
Gaps}(5) = {\rm Gaps}_0(5) = \{0,1 , 2\}$.
\end{cor}

\begin{proof} We know that \eqref {eq:1} holds for any $n \geqslant 5 = d$.
When $2 \leqslant n \leqslant 4$,
\eqref {eq:3} also holds. Thus $G_5\le g_{5,1}-\ell_{5,1}-1=2$ (see
Remarks \ref {rem:nongap} and \ref{rem:est}), and so ${\rm
Gaps}(5) \subseteq \{0,1 , 2\}$. On the other hand,  ${\rm
Gaps}_0(5) = \{0,1 , 2\}$ by Xu's theorem cited above.
\end{proof}

\section{Gaps}\label{sec:gaps}

By Proposition \ref{prop:gaps3} and
Corollaries  \ref{cor:g4} and \ref{cor:5},  we can focus on $d
\geqslant 6$. By  Theorem \ref{prop:nogaps},  ${\rm Gaps}(d)$ is
finite. Hence there exists an integer $n_d \geqslant 0 $ such that
\begin{equation}\label{eq:totalgaps}
{\rm Gaps}(d) = \bigcup_{j=0}^{n_d} {\rm Gaps}_{j}(d), \; \;
\mbox{with} \; \;{\rm Gaps}_{j}(d) : = [a_j, \; b_j],
\end{equation} where
\[
a_0=0<b_0= \frac{d(d-3)}{2} - 3\,\,\, {\rm and}\,\,\,
b_{j-1}+1<a_j\leqslant b_j \,\,\, \mbox{for all}  \,\,\, j>0.
\]
The disjoint and \emph{separated} intervals ${\rm Gaps}_{j}(d)$
are called the
{\em gap intervals}. The initial gap interval ${\rm Gaps}_0(d)$ is
as in (\ref{eq:initialgaps}). Our aim is to determine the next
gap interval ${\rm Gaps}_1(d)$, see Theorem \ref{thm:1gaps} below.

\begin{rem}\label {rem:i} By
 Theorem \ref{prop:nogaps} we have
\begin{equation}\label{eq:gaps-complement}
{\rm Gaps}(d)\subset\NN\cup\{0\}\setminus\bigcup_{n\ge 1} J_n(d).
\end{equation}
Looking at the proof of this theorem, one might guess that there
are $d$--gaps $g$ with
\begin{equation}\label{eq:gaps}
g_{d,n-1}+1\leqslant g\leqslant g_{d,n}-\ell_{d,n}-1
\end{equation} for any $n \geqslant 2$, any time \eqref {eq:1}
does not hold, namely that
\begin{equation}\label{eq:intervalgaps}
I_{n-1} (d):= [g_{d,n-1}+1,  \, g_{d,n}-\ell_{d,n}-1] \subseteq
{\rm Gaps}(d)
\end{equation} if $n \geqslant 2$, $d \geqslant 6$ and  \eqref{eq:1} does not hold.
However, this is not true in general, as \cite [Examples 1.1 and
1.2]{CC} show. For instance, $I_2(20)\not\subseteq {\rm Gaps}(20)$, by \cite [Example 1.1]{CC}. Nonetheless,
$I_1(d)\subseteq {\rm Gaps}(d)$, see Theorem
\ref{thm:1gaps} below.
\end{rem}

Determining all $d$--gaps for $d\geqslant 6$ is a  tricky problem.
 In this section  we show  that there are $d$--gaps other than
the ones in ${\rm Gaps}_0(d)$, i.e. $n_d>0$ for $d\geqslant 6$.
Recall first the following results.

\begin{thm} \label{thm:cl} {\rm (\cite[Thm. (1.2)]{cl}, \cite {cle},
\cite [Thm.  2.1]{Xu1})} Let $X$ be a general surface of degree $d
\geqslant 5$ in $\proj^ 3$, and let $C\in \Lcal_{X,n}$ be an
irreducible curve of geometric genus $g$. Then $$g>\frac {nd(d-5)}
2+1.$$
\end{thm}

\begin{prop} \label{acz2} {\rm (\cite[Cor.  2.9]{cz})}
Let $X$ be a general surface of degree $d\geqslant 3$ in $\PP^ 3$.
If $g\geqslant 0$ and $n\in\{1,2\}$ are such that $\mathcal
V_{n,g}$ is non--empty, then
\begin{equation}\label{eq:ge}
 g_{d,1}-3 \leqslant g\leqslant g_{d,1}
\quad \text {if} \quad n=1 \quad \text {and} \quad
g_{d,2}-9\leqslant g\leqslant g_{d,2} \quad \text {if} \quad
n=2\,.
\end{equation}
\end{prop}

As a consequence, recalling \eqref{eq:intervalgaps}, we have:

\begin{prop}\label{prop:gaps} If $d\geqslant 9$ then
\begin{equation}\label{eq:g}
{\rm Gaps}_1 (d)=I_1(d) = \left[\frac{d^2-3d+4}{2}, \; d^2 - 2d -
9\right].
\end{equation}
Furthermore, for $d=7,8$ we have
\begin{equation}\label{eq:g-prim}
[g_{7,1}+1, \; g_{7,1} + 7] \; \subseteq \; {\rm Gaps}_1 (7)
\,\,\, {\rm and} \,\,\, [g_{8,1}+1, \; g_{8,1} + 16]  \; \subseteq
\; {\rm Gaps}_1 (8).
\end{equation}
\end{prop}

\begin{proof} Let $X$ be a very general surface of degree $d\geqslant 9$
in $\proj^ 3$.  Let $C\in \Lcal_{X,n}$ be an irreducible curve
with geometric genus $g\leqslant g_{d,2}-\ell_{d,2}-1=d(d-2)-9$. By
Theorem \ref {thm:cl}, we have
\begin{equation}\label{eq:gaps2}
d(d-2)-9\geqslant g> \frac {nd(d-5)} 2+1,\,\,\, \text {i.e.,}
\,\,\, n< \frac {2(d^ 2-2d-10)}{d(d-5)}.
\end{equation}Suppose that $n\geqslant 3$.  Then \eqref {eq:gaps2} yields
$d^ 2-11d+20<0$, which implies $d\leqslant 8$, a contradiction.
Thus $n\leqslant 2$
and by Proposition \ref {acz2} one has
$I_1(d) \subseteq   {\rm Gaps}_1 (d)$. Since $\ell_{d,2} = 9$, by
Theorem \ref{thm:cc} we have $\mathcal V_{2,g_{d,2} - 9} (X) \neq
\emptyset$, i.e., $g_{d,2} - 9$ is a $d$--non--gap. Since
$g_{d,1}$ is also a $d$--non--gap,
the equality in (\ref{eq:g}) follows.

 For $d=7,8$ the argument is similar; we leave the details to the
reader. \end{proof}


\section{The  second  gap interval}\label{sec:conjecture}

In this section we  extend Proposition \ref{prop:gaps}, proving
the following theorem.

\begin{thm}\label{thm:1gaps} For all $d\ge 6$ one has
\[ {\rm Gaps}_1(d) = I_1(d)=
\left[\frac{d^2-3d+4}{2}, \; d^2 - 2d - 9\right].\]
\end{thm}

By Proposition \ref{prop:gaps}, we may assume in the sequel that
$6\leqslant d\leqslant 8$.
The proof  consists in a  case by case analysis, which we
will perform in the rest of this section.

\subsection{A reduction step} \label{ssec:red} The following lemma reduces
the analysis to finitely many cases:

\begin{lem}\label{lem:gaps} Assume $6\leqslant d\leqslant 8$.
To prove that $I_1(d)= {\rm Gaps}_1 (d)$,
 it suffices to show that for $X\in \mathcal L_d$ general,
 one has $\mathcal V_{n,g}(X)=\emptyset$ if
\begin{eqnarray}\label{eq:rescases2}
d=6, & n=3, & g\in [11,15]\nonumber \\
d=6, & n=4, & g\in [14,15] \nonumber \\
d= 7, & n=3, & g\in [23,26]  \\
d = 8, & n = 3, & g\in [38,39].\nonumber
\end{eqnarray}
 \end{lem}

\begin{proof}
 For $d=6,7,8$, from (\ref{eq:gaps-complement}) we derive the
inclusions ${\rm Gaps}_1(d)\subset I_1(d)$, where
$I_1(6)=[11,15]$, $I_1(7)=[16,26]$, and $I_1(8)=[22,39]$, see
(\ref{eq:intervalgaps}).  To show the inverse inclusions, we have
to check that every $g\in I_1(d)$ is a $d$-gap for $d=6,7,8$.

Suppose $\mathcal V_{n,g}(X)\not =\emptyset$ for $X\in \mathcal
L_d$ general. Since the intervals in (\ref{eq:ge}) do not meet
$I_1(d)$, from Proposition \ref{acz2} and $d \geqslant 6$ it
suffices to restrict to $n \geqslant 3$.  On the other hand, if $n
\geqslant 5$, \eqref{eq:gaps2} gives $ 5 \leqslant n < \frac {2(d^
2-2d-10)}{d(d-5)}$, i.e., $3d^2 - 21 d + 20 < 0,$ which
contradicts $d \geqslant 6$. Thus it is enough to consider $n \in
\{3,4\}$.
The remaining possibilities are as follows.\\
\begin{inparaenum} [$\bullet$]
\item$(d,n) = (6,3)$ and $g\in I_1(6)$
as in the first line of \eqref{eq:rescases2};\\
\item $(d,n) = (7,3)$; then
by (\ref{eq:g-prim}), $[16,22]\subseteq {\rm Gaps}_1(7)$. Hence it
remains to eliminate
the values of $g$ as in the third line of \eqref{eq:rescases2}; \\
\item $(d,n) = (8,3)$; using again \eqref{eq:g} and  (\ref{eq:g-prim}),
one reduces to the last line of  \eqref{eq:rescases2};\\
\item if $n = 4$, \eqref{eq:gaps2} yields $ 2 d^2 - 16 d + 20 <0$.
So the only possibility is $d=6$.
From Theorem \ref{thm:cl} we deduce $g > 13$, which leaves the
range of $g$ as in the second line of \eqref{eq:rescases2}.
\end{inparaenum}
\end{proof}

\subsection{Strategy of the proof} \label{ssec:strategy}
By Lemma \ref {lem:gaps} we need to show that $\mathcal
V_{n,g}(X)=\emptyset$ for $X\in \mathcal L_d$  general and $d$,
$n$ and $g$  as in \eqref{eq:rescases2}. A basic ingredient will
be the following result.

\begin{prop}\label{prop:acz} {\rm (see \cite[Prop. 2.8]{cz})}
Let $S$ be a smooth projective surface,  $\mathcal H$ the Hilbert
scheme of curves on $S$, and $\mathcal V_g$ the locally closed
subset of $\mathcal H$ formed by irreducible curves of geometric
genus $g$. For any component $\mathcal V \subseteq \mathcal V_g$
we set $$v:=\dim(\mathcal V) \;\;\, \rm{and} \,\;\;
\kappa:= K_S \cdot \Gamma,$$where $\Gamma$ corresponds  to a general
point in $\mathcal V$. Then $v\leqslant
v_0:=\max\{g,g-1-\kappa\}$.
\end{prop}

Let $d$, $n$ and $g$ be as in \eqref{eq:rescases2}. Let $\cF
\subseteq \mathcal L_n$ be an irreducible closed subset, which is
the parameter space for a flat family of surfaces
 in $\PP^3$ of degree $n$.
 We assume that its general point corresponds to an irreducible surface $\Sigma$.

Consider the incidence relation $I\subseteq U_d \times \cF$,
consisting of all pairs $(X,\Sigma)$ such that $X$ and  $\Sigma$
intersect along a reduced, irreducible  curve $\Gamma$ of
geometric genus $g$. Then  $I$ is locally closed  with  the  projections
$$p: I\to
\mathcal U_d \;\;\; {\rm and} \;\;\;q: I\to \cF.
$$
 If $I'$ is an irreducible component of $I$ which dominates $\cF$
via $q$, then
\begin{equation}\label{eq:dimest}\dim(I') = \dim(\cF) + \dim(q^{-1}(\Sigma)\cap I').
\end{equation}
The main point in our strategy is to estimate
$\dim(q^{-1}(\Sigma)\cap I')$. If $(X,\Sigma)\in I'$ and $C$ is
the intersection of $X$ and $\Sigma$, one has a linear system  of
dimension $N_{d-n}+1$ of surfaces of degree $d$ containing $C$.
Thus we get a  $(N_{d-n}+1)$--dimensional set of pairs
$(X,\Sigma)\in I'$ such that the intersection of $X$ and $\Sigma$
is $C$.

Let $\mathcal V_{d,g}(\Sigma)$ be the locally closed subset of
$\mathcal L_{\Sigma,d}:=\vert \mathcal O_{\Sigma}(d)\vert$  formed
by irreducible curves of geometric genus $g$, and let $\mathcal V
\subseteq \mathcal V_{d,g}(\Sigma)$ be any of its irreducible
component. Applying Proposition \ref{prop:acz} to the minimal
desingularization
\begin{equation}\label{eq:mindesing}
\pi: S \to \Sigma \subset \PP^3
\end{equation}
of $\Sigma$, we
obtain the bound
\begin{equation}\label{eq:v0}
\dim(\mathcal V) \leqslant v_0.
\end{equation}
Therefore, for $\Sigma \in \cF$ general,
\begin{equation}\label{eq:CZ0}
\dim(q^{-1}(\Sigma)\cap I') \leqslant v_0 + N_{d-n}+1.
\end{equation}
Hence by (\ref{eq:dimest}),
\begin{equation}\label{eq:CZ00}
\dim(I') \leqslant \dim(\cF) + v_0 + N_{d-n}+1.
\end{equation}

To prove  that
$\mathcal V_{n,g}(X)=\emptyset$  for $X\in  U_d$ general, one
needs to prove that $p|_{I'}$ is not dominant onto $U_d$, i.e.,
\begin{equation}\label{eq:CZ1}
\dim(I') < N_d
\end{equation}
for any $I'$ as above.
Thus (\ref{eq:CZ00}) yields the following sufficient condition for
\eqref{eq:CZ1} to hold:
\begin{equation}\label{eq:CZ1bis}
\dim(\cF) + v_0 < \Phi(n,d):=N_d - N_{d-n} - 1=
\left\{\begin{array}{cc}
\frac{3}{2}d(d+1) & \mbox{if} \; n=3 \\
 & \\
2d^2 +1 & \mbox{if} \; n=4
\end{array}
\right.
\end{equation}for all pairs $(d,n)$ as in \eqref{eq:rescases2}.

The proof of Theorem \ref {thm:1gaps} reduces to check
\eqref{eq:CZ1bis}  for all possible parameter spaces $\cF$ of
cubics (resp., of quartics) in $\PP^3$, whose general element
$\Sigma$ is irreducible.

\subsection{The cubic case}\label{sec:cubic}
Classification of irreducible cubic surfaces in $\PP^3$ started a
century  and a half ago by Schl\"afli in \cite{Sch} and Cayley in
\cite{Cay}, see e.g., \cite{Dol} for a historical account and
references.  About one hundred  years later, Bruce and Wall
(see~\cite{Bru, BruWal}) reconsidered this classification via the
modern theory of singularities.

\begin{desc}\label{desc1} Let $\Sigma \subset \PP^3$ be
an irreducible cubic surface. Then:  \\
\begin{inparaenum}
\item[(i)]  either $\Sigma$ has at most  Du Val  singularities, or  \\
\item[(ii)] it is a cone over a plane cubic,  or \\
\item[(iii)] it is a scroll which is not a cone.
\end{inparaenum}

Case (i) occurs for a general $\Sigma\in \mathcal L_3$; recall
that $N_3=\dim (\mathcal L_3)=19$, see \eqref{400}.

In case (ii) we will denote by $\mathcal C$  the irreducible
closed subvariety of $\mathcal L_3$, whose general point
corresponds to a cone over a smooth, plane cubic. Clearly
$\dim(\mathcal C)=12$.

In case (iii) we will denote by $\mathcal R$ the irreducible
closed subvariety of $\mathcal L_3$, whose  general point
corresponds to an irreducible scroll $\Sigma$ which is not a cone.
Such a scroll $\Sigma$ appears as  the general projection in
$\mathbb P^ 3$ of a smooth rational normal scroll $S$  of degree 3
in $\mathbb P^ 4$, and $\Sing(\Sigma)$ is a  double line
(see~e.g.,\;\cite{BruWal}). An easy parameter count, which can be
left to the reader, shows that $\dim(\mathcal R)=13$.
\end{desc}

We keep the notation of \S\;\ref {ssec:strategy}.  To prove that
$\mathcal V_{3,g}(X)=\emptyset$ for $X\in \mathcal L_d$  general
and $d$ and $g$ for all cases with $n=3$, i.e., as in the first, third and fourth line of
\eqref{eq:rescases2},  we will show that \eqref {eq:CZ1bis}
holds. One has to analyse  cases (i)---(iii) of
Description \ref {desc1} occurring for $\Sigma$ corresponding to
the general point of $\mathcal F$.

We will denote by $H$ the hyperplane section class of $\Sigma$. By
abuse of notation we will also denote by $H$ its total transform
on the minimal desingularization $S$ of $\Sigma$ as in  \eqref
{eq:mindesing}.

\subsubsection{Case (i)} We have $\omega_{\Sigma} \cong \Oc_{\Sigma}(-H)$
 and  $K_S = \pi^*(K_{\Sigma}) = - H$
 (see, e.g., \cite[Lemma 1.2.2]{BeSo}).

Let $\mathcal V\subseteq \mathcal V_{d,g}(\Sigma)$ be an
irreducible component, let $C$ be the curve corresponding to its
general point, and let $\Gamma$ be the proper transform of $C$ on
$S$. Then $\Gamma \sim d H - D$, where $D \geqslant 0$ is a
$\pi$--exceptional divisor, which takes into account if $C$ passes
through some of the singularities of $\Sigma$. Since every
irreducible component of $D$ is a $(-2)$-curve, $D\cdot K_S=0$.
Hence  $\kappa = K_S \cdot \Gamma = - 3d$.

By \eqref {eq:v0} and Proposition \ref{prop:acz} we have
$\dim(\mathcal V) \leqslant v_0 = 3d + g-1$. Since  $\dim(\cF)
\leqslant 19$, a sufficient condition for \eqref {eq:CZ1bis} to
hold is
\begin{equation}\label{eq:utile}
3d^2 - 3d - 36 > 2g.
\end{equation}
Then \eqref {eq:utile} holds for all  the $(d,g)$ in
\eqref{eq:rescases2} which correspond to $n=3$. Hence the same is
true for \eqref {eq:CZ1bis}.

\subsubsection{Case (ii)}
One has  \begin{equation}\label{eq:Fconegen} \dim(\cF) \leqslant
\dim(\mathcal C) = 12.
\end{equation}
Let $Y \subset \PP^2$ be the plane cubic which is the base of the
cone $\Sigma$. There are the following possibilities:
\begin{inparaenum}[(a)]
\item $Y$ is smooth;
\item $Y$ is nodal;
\item $Y$ is cuspidal.
\end{inparaenum} We will discuss cases (a) and (b)  only, since (c) is
similar to  (b) and can be left to the reader.

\subsubsection{Case (ii,a)}\label {ssec:Case (ii,a)}
The minimal desingularization of $\Sigma$ is   $S = \PP_Y(\mathcal
\cO_Y \oplus \cO_Y(1))$, which is a ruled surface with base $Y$.
We denote by $F$ the numerical equivalence class of a fibre of the
structure morphism $\sigma: S\to Y$ and by $E$ the section with
$E^ 2=-3$ which is contracted to the vertex $\mathfrak v$ of
$\Sigma$. One has $H\equiv E+3F$ and $K_S \equiv -E-H$. With the
usual notation,
let $\Gamma$ be the proper transform of $C\in \mathcal V_{d,g}(\Sigma)$ on $S$. \\
\begin{inparaenum}
\item[(\dag)] If $C$ does not pass through $\mathfrak v$, then $\Gamma \equiv dH$,
$\kappa = K_S \cdot \Gamma = - 3d$, and by Proposition
\ref{prop:acz} we find the upper bound $v_0 = 3d+g-1$ in \eqref
{eq:v0}.
Then the discussion proceeds as in case (i) above, with the same conclusion.   \\
\item[(\ddag)] If $C$ passes through $\mathfrak v$, then
one has $\Gamma\equiv dH-E$. Indeed, a priori one has $\Gamma \equiv dH - mE$. On the other hand, since $X$ is smooth at 
$\mathfrak v$, the general ruling of the cone intersects $X$ at $d-1$ points off $\mathfrak v$; 
hence $\Gamma \cdot F = d-1$, which proves $m=1$ (the same holds in all cases below, 
dealing with cones).

Thus $\kappa = K_S \cdot \Gamma = - 3d - 3$.
By Proposition \ref{prop:acz}, we find $v_0 = g + 3d + 2$. Taking
into account \eqref {eq:Fconegen} and proceeding as in case (i),
one sees that \eqref{eq:CZ1bis} holds in all cases of
\eqref{eq:rescases2} with $n=3$.
\end{inparaenum}

\subsubsection{Case (ii,b)} As before, \eqref {eq:Fconegen} holds.
The minimal desingularization $S$ of $\Sigma$ is the Hirzebruch
surface $\mathbb{F}_3 = \PP(\cO_{\PP^1} \oplus \cO_{\PP^1}(-3))$.
We denote again by $F$ the numerical equivalence class of a fibre
of the structure morphism $\sigma: S\to \PP^ 1$ and by $E$ the
section with $E^ 2=-3$ which is contracted to the vertex
$\mathfrak v$ of $\Sigma$. One has $H\equiv E+3F$ and  $K_S \equiv
- 2 H + F$.

Let $\Gamma$ and $C$ be as usual. Two cases have to be analyzed. \\
\begin{inparaenum}
\item[(\dag)] If $C$ does not pass through
$\mathfrak v$, then $\Gamma \equiv dH$.
One has $\kappa = K_S \cdot \Gamma = - 5d$ and  $v_0= 5d + g -1$.
By  \eqref {eq:Fconegen}, to prove that \eqref{eq:CZ1bis} holds,
it suffices to show that $3d^2 - 7d - 22 > 2g$ for all $d$ and $g$
in \eqref{eq:rescases2} with $n=3$.
A direct computation confirms that this is  the case. \\
\item[(\ddag)] $C$  passes through $\mathfrak v$. As in case (ii,a)--(\ddag),
one has $\Gamma \equiv dH - E$.   Thus, $\kappa = K_S \cdot \Gamma
= - 5d -1$ and similar computations as in the previous case show
that \eqref{eq:CZ1bis} holds.
\end{inparaenum}

\subsubsection{Case (iii)} One has
\begin{equation}\label{eq:Fscroll}
\dim(\cF) \leqslant \dim(\mathcal R) = 13.
\end{equation}

The minimal desingularization $S$ of $\Sigma$ is isomorphic to
$\mathbb{F}_1$, with the structure map $\sigma: S\to \PP^ 1$. We
denote by $E$ the section of $S$ with $E^ 2=-1$ and with $F$ a
fibre. Then $H\equiv E+2F$ and  $K_S \equiv - 2E - 3 F$. Since
$\Gamma \equiv dH$, we get $\kappa = K_S \cdot \Gamma = - 5d$,
hence we find $v_0 = 5d + g - 1$ as in case (ii,b)--(\dag).   By
\eqref{eq:Fscroll}, to prove that \eqref{eq:CZ1bis} holds it
suffices to verify that  $3d^2 - 7d - 24 > 2g,$ for all $d$ and
$g$ in \eqref{eq:rescases2} with $n=3$.
 A direct check shows that this is the case.

\vskip 15pt

In conclusion, the above analysis shows that, for $X\in U_d$
general, one has $\mathcal V_{n,g}(X)=\emptyset$ for $d,n, g$ as
in \eqref {eq:rescases2} with $n=3$.

\subsection{The quartic case}\label{sec:quartic}
The only case left  from \eqref{eq:rescases2} is $d=6, n=4, g\in
[14,15]$.  To finish the proof of Theorem \ref {thm:1gaps} we have
to verify that for $X\in U_6$ general, one has $\mathcal
V_{4,g}(X)=\emptyset$  for $g\in \{14,15\}$.

We will keep the notation as in \S \ref {ssec:strategy}. From
\eqref{eq:CZ1bis},  one has $\Phi(4,6) = 73$ and $\dim(\cF)
\leqslant 34 = \dim (\mathcal L_4)$. Therefore, for
\eqref{eq:CZ1bis} to hold, it suffices to prove the upper bound
\begin{equation}\label{eq:CZquartic}
v_0 \leqslant  39.
\end{equation} This is what we will do for all cases discussed below,
except the last one, where the argument is different.

\subsubsection{Classification of quartic surfaces}
The classification of  irreducible quartic surfaces in $\PP^3$ is
as old as that of cubics, see e.g., \cite{Dol}. Similarly as for
cubics, we will use a modern version elaborated in
\cite{IsNa,Um,Um2,Ur}, which we shortly recall here. For any such
quartic $\Sigma$ one has $\omega_{\Sigma} \cong \Oc_{\Sigma}$. As
usual, we let $\pi: S \to \Sigma$ be the minimal
desingularization, and we  keep the notation as in the cubic case.

First we treat the case $\Sigma$ normal (see \cite{Um,Um2}).  If
$p \in {\rm Sing}(\Sigma)$, the {\em geometric genus} of $p$ is
defined to be
\[
p_g(p)=\dim_{\mathbb{C}}((R^1\pi_*\cO_S)_p)
\]
(see \cite{Wag} or \cite[Def.\;1]{Um})).
We set$${\rm Irrat}(\Sigma) := \{ p \in
{\rm Sing}(\Sigma) \; | \; p_g(p) >0\},$$which is the set of {\em
irrational singularities} of $\Sigma$. 

\begin{prop}\label{prop:Umezu}{\rm (cf. \cite[Propositions 5,7,8,\;Theorem 1]{Um})}
Let $\Sigma \subset \PP^3$ be a normal, irreducible quartic surface. One has:\\
\begin{inparaenum}
\item [(a)] if $p\in {\rm Sing}(\Sigma)$ and $p_g(p)=0$, then $p$ is a Du Val singularity;\\
\item [(b)] there exists a unique effective divisor $E$ on $S$ such that
$\cO_S(E) \cong \omega_S^{\vee}$.
\end{inparaenum}

Moreover:\\
\begin{inparaenum}
\item [(i)] if ${\rm Irrat}(\Sigma) = \emptyset$, then $S$ is a $K3$ surface, i.e., $E=0$;\\
\item [(ii)] if ${\rm Irrat}(\Sigma) \neq \emptyset$, then $S$ is birationally
equivalent to a ruled surface, $E > 0$ and its connected components bijectively
correspond via $\pi$ to singularities in ${\rm Irrat}(\Sigma)$.\end{inparaenum}

Furthermore, if $q := h^1(\cO_S)$, then: \\
\begin{inparaenum}
\item[($ii$-1)] if $q \neq 1$, then ${\rm Irrat}(\Sigma)$ consists of  a single
 point $p$ such that $p_g(p) = q+1$;\\
\item[($ii$-2)] if $q=1$, then   ${\rm Irrat}(\Sigma)$ consists either of one point $p$,
with  $p_g(p) = 2$, or of two points
$p_i$, for $i=1,2$, with $p_g(p_i)= 1$, that are both {\em
simple elliptic}, i.e., $\pi^{-1}(p_i)$ is a smooth, irreducible
elliptic curve.
\end{inparaenum}
\end{prop}

In case (ii) of Proposition \ref {prop:Umezu}, i.e., when $\Sigma$
is a normal quartic surface in $\proj^3$ with irrational singular
points, there is a detailed classification in \cite{IsNa}, which
we will need to go through later.
 It can be briefly summarized as follows.

\begin{prop}\label{prop:IsNa} Let $\Sigma \subset \PP^3$ be a normal,
irreducible quartic surface such that ${\rm Irrat}(\Sigma) \neq \emptyset$. Then either\\
\begin{inparaenum}
\item [(i)] $q=0$ and $\Sigma$ is rational, or\\
\item [(ii)] $q=1$ and $\Sigma$ is birational to an elliptic ruled surface or\\
\item [(iii)] $q=3$ and $\Sigma$ is a cone over a smooth plane quartic curve.
\end{inparaenum}
\end{prop}

As for the non--normal case, we have:

\begin{prop}\label{prop:Umezu2}{\rm (cf.  \cite[Lemma\;2.3]{Ur})}
Let $\Sigma \subset \PP^3$
be a non--normal, irreducible quartic surface. Then $S$ is either
a scroll over a smooth curve of genus $2$, or an elliptic scroll
or a rational surface.
\end{prop}

We will examine the various cases, first treating the normal, then
the non--normal ones. Remember that to accomplish the proof of
Theorem \ref {thm:1gaps} it suffices to establish inequality
\eqref {eq:CZquartic}.

\subsubsection{The $K3$ case} This is case (i) in Proposition \ref {prop:Umezu}. Then
$v_0 = g \leqslant 15$ (cf. Proposition \ref{prop:acz}). So \eqref
{eq:CZquartic} holds.
\bigskip

\subsubsection{Normal quartic surfaces with an irrational singularity}
Next we turn to case (ii) in Proposition \ref {prop:Umezu}, which,
according to Proposition \ref {prop:IsNa}, gives rise to various
subcases. We refer to \cite{IsNa} for details. To make the reading
more accessible,
 let us first overview
the terminology and the main classification principle in
\cite{IsNa}. The latter uses the triplets 
$(X,B,G)$ consisting in a smooth, projective surface $X$, a smooth
non-hyperelliptic curve $B$ of genus 3 on $X$, and an effective
anticanonical divisor $G\in |-K_X|$, where $G\neq 0$. Such a
triplet {\em satisfies condition} $\mathcal{C}_r$ if $K_X+B$ is
nef and $B\cdot G=r$. If $r\ge 1$, then blowing up $\sigma: X'\to
X$ at a point of $B\cap G$ leads to a $\mathcal{C}_{r-1}$-triplet
$(X', B', G')$ with $B'=\sigma^*(B)-F$ and $G'=\sigma^*(B)-F$,
where $F$ is the exceptional $(-1)$-curve. After $r$ blowups one has a birational morphism
$\rho: S\to X$ and one
arrives at a $\mathcal{C}_0$-triplet  $(S,H,E)$ where
\begin{equation}\label{eq: separation} 
H := \rho^*(B) - \Delta\quad\mbox{and}\quad E : = \rho^*(G) - \Delta,
\end{equation} where $\Delta$ is the total $\rho$--exceptional divisor.
This process is called {\em separation} (\cite[p.\;947]{IsNa}) and $(S,H,E)$ is
called the (result of the) separation of $(X, B, G)$. Notice that the
divisor $E$ can be reducible and/or non-reduced, even if $G$ is
reduced and irreducible. This may happen if $G$ is singular
and $r\ge 1$.

One says that a $\mathcal{C}_0$-triplet $(S,H,E)$  is a {\em basic
triplet} (our terminology here slightly differs from the one in \cite{IsNa}),  if
$H$ meets every $(-1)$-curve on $S$. There exists a classification
of all basic triplets into 4 types $A,B,C,D$ (\cite
[Theorem 1.7]{IsNa}), together with a list of examples of each
type (\cite [\S 2]{IsNa}) obtained via separation, that we will
permanently address below. The main theorem in  \cite[\S 3]{IsNa}
asserts that this list is exhaustive, and so describes all the
normal quartic surfaces in $\proj^3$ with irrational
singularities. Together with \cite[Prop.\;1.4]{IsNa},
this yields the following theorem.

\begin{thm}\label{thm: C-triplets} Any basic triplet
$(S,H,E)$  arises as the minimal
desingularization $\pi=\varphi_{|H|}\colon S\to\Sigma$ of a normal
quartic surface $\Sigma$   in $\proj^3$ with irrational singular
points, where $H$ is the pullback of a hyperplane section of
$\Sigma$ which does not pass through any irrational singular
point, and $E$ is an effective $\pi$-exceptional anticanonical
divisor on $S$.
Conversely, the minimal desingularization $\pi: S\to\Sigma$ of a
normal quartic surface $\Sigma$ in $\proj^3$ with an irrational
singular point yields a  basic triplet $(S,H,E)$, with $H$ and $ E$ as before.
\end{thm}


Keeping notation as in \S\;\ref{ssec:strategy}, let $C$ be an irreducible curve on $\Sigma$ cut out by a smooth sextic surface
$X$ in $\proj^3$,  and let $\Gamma$ be its proper transform on $S$. Then $\Gamma\sim 6H-D$, where $D$ is an effective
$\pi$-exceptional divisor on $S$. We write $D=D_E+D'$, where $D_E$
is supported on ${\rm Supp}(E)$ and is contracted to ${\rm
Irrat}(\Sigma)\cap C$, whereas $D'$ is contracted to the Du Val
singularities of $\Sigma$ situated on $C$. From $C\sim 6H$ and
$E\cdot H=E\cdot D'=0$ we deduce
\begin{equation}\label{eq:DE}
-\kappa = - K_S\cdot\Gamma = E\cdot (6H-D)=-E\cdot D=-E\cdot D_E,
\end{equation}i.e the presence of Du Val singularities does not
affect $\kappa$ (this will be used in all cases discussed below). Thus, \eqref{eq:CZquartic} reads
$$v_0=\max\{g,g-\kappa-1\}=\max\{g,g-E\cdot
D_E-1\}\le 39\,.$$ Since $g\in\{14,15\}$, (\ref{eq:CZquartic})
follows once
\begin{equation}\label{eq:CZquartic-bis} -E\cdot
D_E\le 25\,.\end{equation} We will check inequality
(\ref{eq:CZquartic-bis}) case by case.

\subsubsection{The normal cone case} This is case (iii) in
Proposition \ref {prop:IsNa}, i.e., $\Sigma$ is the cone with
vertex $\mathfrak v$ over a smooth quartic $Y \subset \PP^2$. Then
$S = \PP(\Oc_Y \oplus \omega_Y)$. If $E_0$ is the section
contracted by $\pi$ to $\mathfrak v$, then $E_0^ 2=-4$,
$E=-K_S=2E_0$, $H\cdot E_0=0$, and $D_E= E_0$ (remember the argument in \S \ref {ssec:Case (ii,a)}, ($\ddag$)). Thus $-E\cdot
D_E= -2E_0^2=8$ and so \eqref{eq:CZquartic-bis}  holds.

%

\medskip

Next we examine case (i) in Proposition \ref {prop:IsNa} (see
\cite {IsNa} for the cases considered below).

\subsubsection{Normal rational quartics:  Case (a)}\label {ssec:casea}
This case is described in \cite [\S~2.2.1]  {IsNa}.

Let $X$ be a weak (or generalized) del Pezzo surface of degree
$2$, i.e., $-K_X$ is nef and big and $K_X^2=2$, see e.g.,
\cite{Dol}.  Then ${\rm Bs}(|-K_X|) = \emptyset$, $\dim (\vert
-K_X\vert )=2$,
 and $\phi_{|-K_X|}: X\to \PP^ 2$ is generically finite, of degree $2$
(see\;\cite[p.\;944]{IsNa}). A general member $B \in |-2K_X|$ is a
smooth, non--hyperelliptic curve of genus $3$ (see \cite[Lemma
2.1]{IsNa}).
If $G\in |-K_X|$,
then $B\cdot G=4$. Thus the triplet $(X,B,G)$ satisfies
condition $\mathcal{C}_4$ and we can consider its  separation $(S,H,E)$ as in \eqref{eq: separation}, 
which is a basic triplet. 
One has  $H^ 2=4$ and $- K_S \equiv  E$, with $E^2= -2$. Note that
$| \rho^*(G)|$ has dimension 2 and is base point free.
Furthermore, $H\cdot  \rho^*(G)=4$.

 According to Theorem \ref{thm: C-triplets},  $S$ is the
minimal desingularization of the normal quartic surface
$\Sigma:=\phi_{|H|}(S)  \subset \PP^3$, and $ \pi=\phi_{|H|}$
contracts $E$ (and no other curve) to a unique irrational singular
point $p \in \Sigma$, with $p_g(p) = 1$ (see Proposition
\ref{prop:Umezu}-($ii$-1)).

%
Let $C \sim 6 H$ be an irreducible curve  on $\Sigma$ of geometric
genus $g$, and let $\Gamma$ be the proper transform of $C$ on $S$.
We have $\Gamma \equiv 6H - D$, with $D = D_E + D'$ as above. 
From \eqref {eq: separation}, $-\kappa=E\cdot \Gamma = (\rho^*(G)-\Delta)\cdot
\Gamma$. Since $\Gamma$ is
irreducible and non-rational, we have $\Delta\cdot \Gamma\geqslant
0$. Hence $E\cdot \Gamma \leqslant (\rho^*(G)-\Delta)\cdot \Gamma\leqslant
\rho^*(G)\cdot \Gamma= \rho^*(G)\cdot (6H - D)$. Since $\rho^*(G)$
is nef, one has $\rho^*(G)\cdot D\geqslant 0$. Thus
$-\kappa=E\cdot \Gamma = \rho^*(G)\cdot (6H - D) \leqslant
\rho^*(G)\cdot (6H) =24$. This proves \eqref {eq:CZquartic-bis}.



\begin{rem} An equivalent description of $\Sigma$ is gotten by taking
the image of $\PP^ 2$ via the rational map determined by a linear system
of curves of degree 6 with 7 double and 4 simple base points all on
a cubic.\end{rem}

\subsubsection{Normal rational quartics:  Case (b)}\label {ssec:caseb}
This case is described in \cite[\S\;2.2.2]{IsNa}.

 Let  $Z$ be a  weak
del Pezzo surface of degree $1$. Then $|-K_Z|$ is a pencil, and
 ${\rm Bs}(|-K_Z|)$ consists of  a single point  $b$ (cf.\;\cite[p.\;944]{IsNa}).  If $L\in |-K_Z|$,
 then $b$ is a smooth point of $L$. Let $L'$ be the irreducible component
 of $L$ containing $b$.
 One can choose a point $q \in L'$ such that:\\
\begin{inparaenum}
\item[(1)] $q$ is a smooth point for $L$, and \\
\item[(2)] $\cO_{L}(b-q)$, $\cO_{L}(2b-2q)$ are not isomorphic to $\cO_{L}$.\\
\end{inparaenum}
Then there exists a unique point $q_1 \in L'$ satisfying
$\cO_{L}(q_1) \cong \cO_{L}(3b-2q)$  and $q_1 \neq b$ by condition
(2).

Let $f :X \to Z$ be the blowup of $Z$ at $q$ with exceptional
divisor $\Xi$,  and let $G$ be the proper transform of $L$ on $X$.
Let $b' = f^{-1}(b)$ and $\{q'\} = G \cap \Xi$. The points $b'$
and $q'$ are contained in the proper transform $G'$ of $L'$. There
is a smooth point $q_1' \in G$ such that $\cO_{G}(3b'-2q') \cong
\cO_{G}(q_1')$. Then $f(q_1') = q_1$, $\{q_1'\} = {\rm Bs}(|3G +
\Xi|)$ and $h^0(X, \cO_X(3G + \Xi)) = 4$.

A general member $B\in |3G + \Xi|$ is a smooth, non-hyperelliptic
curve of genus $3$, and $(X,B,G)$ is a  triplet satisfying
condition $\mathcal{C}_1$, see \cite[Lemma 2.2]{IsNa}. The
separation of $B$ and $G$ consists in blowing--up $\rho : S \to X$
at $q_1'$ with exceptional divisor $\Delta$. Letting $H$, $E$, and
$\Xi'$ be the proper transforms of $B$, $G$ and $\Xi$,
respectively, we get a basic triplet $(S,H,E)$ with
$H\cdot E=0$, $H^ 2=4$, $K_S=-E$, and $E^ 2=-1$.

We let $\Lambda$ be the total transform of $L$ on $S$. Then
$\Lambda$ is nef, $|\Lambda|$ is a pencil with $\Lambda^2=1$, with
a single base point, $\Lambda\cdot H=4$, and
$E=\Lambda-\Xi'-\Delta$, where $\Lambda\cdot \Xi'=\Lambda\cdot
\Delta=0$.

One has ${\rm Bs}(|H|) = \emptyset$, and $\pi=\phi_{|H|}: S \to
\PP^3$ is the minimal desingularization of the quartic
$\Sigma=\pi(S)$, which contracts $E$ (and only this curve) to an
irrational singular point $p\in {\rm Irrat}(\Sigma)$
 with $p_g(p)=1$.

For an irreducible curve $C \sim 6 H$ on $\Sigma$ of geometric
genus $g\in\{14,15\}$, we let as before $\Gamma$ be the proper
transform of $C$ on $S$.
 Then $\Gamma \equiv 6H - D$, with $D= D_E + D'$ (cf. \eqref{eq:DE}).  Since $\Gamma$ is
non--rational, we have $\Xi'\cdot\Gamma\geqslant 0$ and
$\Delta\cdot\Gamma\geqslant 0$. Furthermore, $\Lambda\cdot
D\geqslant 0$ since $\Lambda$ is nef. Hence
$$-\kappa=E\cdot\Gamma=(\Lambda-\Xi'-\Delta)\cdot\Gamma\leqslant\Lambda\cdot\Gamma
=\Lambda\cdot (6H - D)\leqslant 6\Lambda\cdot H=24\,.$$ Thus again
\eqref {eq:CZquartic-bis} is satisfied.

%
%
%
%

\begin{rem}
The quartic $\Sigma\subset\proj^3$ in this case is the image of
$\PP^ 2$ under the rational map determined by a linear system of
curves of degree 9 with 8 triple and one double base points, all
on a cubic.
\end{rem}

\subsubsection{Normal rational quartics:  Case (c)} This case is described in \cite[\S\;2.2.3]{IsNa}.

Consider the Hirzebruch surface $\mathbb{F}_1=\PP(\cO_{\PP^1}
\oplus \cO_{\PP^1}(1))$.
 Let $\Xi$ be the  section with $\Xi^ 2=-1$ and $F$ a ruling. Fix a point $x_0\in \Xi$,
 and let $F_0$ be the ruling containing $x_0$.
There exists a  reduced divisor   $\Delta \in |2 \Xi + 6 F|$ such that: \\
\begin{inparaenum}
\item[(1)] if $f: V \to \mathbb{F}_1$ is the double covering branched  along $\Delta$,
then $V$ has only Du Val singularities; \\
\item[(2)] $\Xi \;\; |\!\!\!\!\! \subset \; \Delta$; \\
\item[(3)] $x_0 \in \Delta \cap \Xi$ and ${\rm mult}_{x_0}(\Delta|_{\Xi}) = 1$; \\
\item[(4)] $F_0 \;\; |\!\!\!\!\! \subset \; \Delta$ and $F_0 \cap \Delta = \{x_0\}$.
\end{inparaenum}

Fixing such a $\Delta$, we let $\lambda :Y \to V$
be the minimal desingularization of the double covering $V$ as in
(1). The surface $Y$ is rational, because it carries the pencil of
rational curves $|\widetilde\lambda^* (F)|$, where
$\widetilde\lambda=f\circ \lambda\colon Y\to\mathbb{F}_1$.  One
has $K_Y \sim \widetilde\lambda^*(K_{\mathbb{F}_1} + \Xi + 3F)
\sim \widetilde\lambda^*(- \Xi)$. Letting $G :=
\widetilde\lambda^*(\Xi)$, by (2) and (3) above, there exists an
irreducible component $G' \subset G$ such that the induced
morphism $G' \to \Xi$ is a double covering, whereas the other
components of $G$ are contracted by $\widetilde\lambda$ to points
of $\Xi$. These components are also contracted to singular points
of $V$, hence by (1), they are rational curves with
self-intersection $-2$.

The morphism $\widetilde\lambda$ is finite over an open
neighborhood of $x_0$ and $\widetilde\lambda^{-1}(x_0)$ consist of
a single point $b' \in G'$. One has
$\widetilde\lambda^*(F_0) = F_1 + F_2$, where $F_1,F_2$ are
$(-1)$--curves  such that $F_1\cdot F_2=1$ and $ F_1 \cap
F_2=\{b'\} $.

Let $\mu :Y \to S$ be the blowdown of $F_1$. Letting
$$E:= \mu_*(G) \sim - K_S, \; \; L := \mu_*(F_2), \;\;{\rm and} \;\; b := \mu(b'),$$
we get  $\mu^*(E) = G + F_1$ and $\widetilde\lambda^*(F) \sim
\mu^*(L)$. By \cite[Lemma\;2.5]{IsNa}, $|L+2E|$ is base point free
and its general member $H$ is a smooth, non--hyperelliptic curve
of genus $3$, with $H^ 2=4$. Then  $\pi=\phi_{|H|} $ maps $S$ to a
normal quartic $\Sigma \subset \PP^3$ and  contracts $E$ (and only
$E$) to a unique irrational singular point $p\in {\rm
Irrat}(\Sigma)$ with $p_g(p)=1$.

Note that $E$ contains a component   $E':=\mu_*(G')$. The other
(possible) components of $E$ are rational curves with
self--intersection $-2$ and so, these have zero intersection with
$K_S\equiv -E$.  Hence $-1=E^ 2=E\cdot E'$. Since $L=\mu_*(F_2)$,
where $L^2=0$, the linear system
$|L|$ is a base point free pencil of rational curves. One has
$L\cdot E=L\cdot E'=2$ (i.e., $L$ has zero intersection with the
components of $E$ different from $E'$),
 and so $L \cdot H=4$.

For $C \sim 6 H$ on $\Sigma$ of geometric genus $g$, we have
$\Gamma \equiv 6H - D$, with $D = D_E +D'$. Let $\alpha$ be
the multiplicity of $E'$ in $D_E$. Then, from \eqref{eq:DE}, $\kappa=E\cdot D_E=\alpha E\cdot E'=-\alpha$. On the other hand,
$2\leqslant L\cdot \Gamma= 24-2\alpha$, thus
$-\kappa=\alpha\leqslant 11$ and so, \eqref {eq:CZquartic-bis}
holds.
%

\subsubsection{Quartic monoids}
This case corresponds to quartic surfaces with a triple point (see
\cite[\S\;2.4]{IsNa}).

Explicitly, let $B$ be a smooth quartic in $\proj^2$ and let $G$
be a cubic in $\PP^2$, not necessarily reduced or irreducible.
Performing the separation of $(\PP^ 2,B,G)$,  we have a basic triplet $(S,H,E)$ with 
$E^2 = -3$ and $K_S\sim-E$. The morphism $\pi=\phi_{|H|}$ sends
 $S$ to a normal quartic $\Sigma \subset \PP^3$, contracting
$E$ to the only irrational (triple) point $p$ of $\Sigma$, with
$p_g(p)=1$. 

Letting $\Lambda=H-E\sim H+K_S$, we obtain a base point free
linear system $|\Lambda|$ of dimension 2, with $\Lambda^2=1$ and
$p_a(\Lambda)=0$. The morphism $\phi_{|\Lambda|}: S\to \PP^ 2$
factors through the (stereographic) projection
$\Sigma\to\proj^2$ with center $p$. In fact, this morphism is
nothing but the  above separation.

For $C \sim 6H$ on $\Sigma$ of geometric genus $g$ we have, as
usual,
$\Gamma \equiv 6H - D$, with $D = D_E + D'$. Since $|\Lambda|$ cuts out on $\Gamma$ a linear series of
dimension 2, we have $4\leqslant \Lambda\cdot \Gamma= (H-E)\cdot
(6H-D)=24+E\cdot D_E$. Hence, from \eqref{eq:DE}, $-\kappa =-E\cdot D_E\leqslant 20$ so \eqref {eq:CZquartic-bis} holds.
\bigskip

Next we turn to case (ii) in Proposition \ref {prop:IsNa}. In
\cite[\S~2.3]{IsNa} (cf.\;also\;\cite{Um2}) there is a
classification, which we will go through. From Proposition
\ref{prop:Umezu}-($ii$), the cardinality of ${\rm Irrat} (\Sigma)$
can be either $1$ or $2$. \medskip

\subsubsection{Ruled elliptic normal quartics:
Case (a)}\label{ssec:rcasea} This  case is described in
\cite[\S\;2.3.2]{IsNa}.

Let $Y$ be a smooth, elliptic curve, and let $q_1, q_2\in Y$ be
such that $2q_1 \;\; |\!\!\!\!\!\sim \;2q_2$. Letting $\mathcal E
= \cO_Y(q_1) \oplus  \cO_Y(q_2)$ and $X = \PP(\mathcal E)$, we
consider the structure morphism $\sigma :X \to Y$ and the fibres
$F_i := \sigma^*(q_i)$, $ i =1, 2$. Let further  $H_{\cE}$ be the
tautological divisor class on $X$
 and $F$ the numerical class of a fibre.

The surface $X$ possesses two sections $\Xi_i \sim H_{\cE} -F_i$,
for $ i =1, 2$. One has $\Xi_1 \cap \Xi_2 = \emptyset$ and $|-K_X|
= \{G\}$, where $G:= \Xi_1 + \Xi_2$.  A general member $B\in
|H_{\cE} - K_X|$ is a smooth, non--hyperelliptic curve of genus
$3$ which intersects transversally $\Xi_i$ at a point $x_i$, $ i
=1, 2$.

Let $\rho: S \to X$ be the blowup of $X$ at the points $x_i$
with exceptional divisors $\Delta_i $, for $ i =1, 2$. Consider
the proper transform $\Xi'_i$ of $\Xi_i$ for $ i =1, 2$, the
divisor $E=\Xi_1' + \Xi_2' \equiv -K_S$, and the proper transform
$H$ of $B$ on $S$.  Then $K_S^2  = -2$,   $H^ 2=4$, $\Xi_i'^2 =
-1$, and $H\cdot \Xi_i'= 0$, for $ i =1, 2$. We abuse notation and
denote by $F$ the total transform on $S$ of a ruling of $X$. One
has  $H\cdot F=3$.

Thus we got a separation $(S,H,E)$ of $B$ and $G$. The linear
system $|H|$ on $S$ is base point free of dimension 3, and
$\pi=\phi_{|H|}$ maps $S$ to a quartic surface $\Sigma$ in
$\proj^3$ with ${\rm Irrat}(\Sigma) = \{p_1, p_2\}$, where $p_i =
\pi (\Xi'_i)$ is a simple elliptic singularity with $p_g(p_i) =1$,
for $i=1, 2$.  Since $H\cdot F=3$, $\Sigma$ is swept out by an
elliptic pencil of rational normal cubics.

Take $C \sim 6H$ on $\Sigma$ of geometric genus $g$.
Then
$\Gamma \equiv 6H - D$, with $D_E=\alpha_1 \Xi'_1 + \alpha_2 \Xi'_2$, $\alpha_1,\alpha_2$ 
non--negative integers.
From \eqref{eq:DE}, one has $-\kappa = \alpha_1 + \alpha_2$. On the
other hand, $\Gamma\cdot F=18- (\alpha_1+\alpha_2)\geqslant 2$,
because $g>1$. Hence $-\kappa\leqslant 16$ and so, \eqref{eq:CZquartic-bis} holds.


\subsubsection{Ruled elliptic normal quartics: Case (b)}
This  case is contained in \cite[\S\;2.3.1]{IsNa}, to which we
refer for  details.

Let $Y$ be a smooth, elliptic curve with a line bundle $A$ of
degree $2$.  Letting  $\cE=\cO_Y \oplus A$ we consider the
elliptic ruled surface $X = \PP(\cO_Y \oplus A)$ with  the
structure morphism $\sigma: X \to Y$. We let $\Xi_1$ denote the
unique section of $\sigma$ with $\Xi_1^ 2=-2$, $H_{\cE}$ the
tautological line bundle, and $F$ the numerical class of a ruling.
One has $h^ 0(X,\cO_X(H_{\cE}))=3$ and $H_{\cE}^ 2=2$,
$H_{\cE}\cdot \Xi_1=0$. Furthermore $-K_X\equiv H_{\cE}+\Xi_1$.
Hence $|-K_X|$ has $\Xi_1$ as a fixed component and $|H_{\cE}|$ as
its movable part.  If $G\in |-K_X|$, then $h^0(G,\cO_G) = 2$.
Letting $G=\Xi_1+\Xi_2$, we note that  either $\Xi_1\cap
\Xi_2=\emptyset$, or $\Xi_2$ consists of $\Xi_1$ plus the sum of
two fibres with class in $|\sigma^*(A)|$.

A general member $B\in |-K_X + \sigma^*(A)|$ is a smooth,
non--hyperelliptic curve of genus $3$ with $B^ 2=8$ and $B \cdot G
=4$. Note that $B\cdot \Xi_1=0$, so $B\cap \Xi_1=\emptyset$. Thus
$(X,B,G)$ is a $\mathcal{C}_4$-triplet.

Performing a separation of $(X,B,G)$,  we obtain a
 basic triplet $(S,H,E)$ with $H$ and $E\sim -K_S$ being
the proper transforms of $B$ and $G$, respectively. We  let $E_i$
denote the proper transforms of $\Xi_i$, for $i=1, 2$. One has
$K_S^ 2=-4$, $H^ 2=4$, $H\cdot E_i=0$, and ${E_i}^ 2=-2$, for
$i=1, 2$. We abuse notation and let $F$ still denote the total
transform of a ruling $F$ on $S$. One has $H\cdot F=2$.

The linear system $|H|$ is base point free of dimension 3, and
$\pi=\phi_{|H|}$ maps $S$ to a normal quartic $\Sigma \subset
\PP^3$. Since $H\cdot F=2$, the surface $\Sigma$ is swept out by
an elliptic pencil of conics  $|F|$. Furthermore, $E=E_1+E_2$ is
contracted to one or two irrational singular points.
More precisely,
 if $\Xi_1\cap \Xi_2=\emptyset$ and so
$E_1\cap E_2=\emptyset$, then $p_i=\pi(E_i)$, $i=1,2$, are two
distinct simple elliptic singularities of $\Sigma$ with
$p_g(p_i)=1$. Otherwise $E_2=E_1+F_1+F_2$, where $F_1,F_2$ are two
(may be coinciding) $(-2)$--curves obtained as a result of two
blowups on each of two (may be coinciding) fibres in
$|\sigma^*(A)|$. In this case $\Sigma$ has a unique irrational
singular point $p$ with $p_g(p) =2$ (see Proposition
\ref{prop:Umezu}-($ii$-2)).

Take a curve $C \sim 6H$ on $\Sigma$ of geometric genus $g$.
According to the cardinality of ${\rm Irrat}(\Sigma)$,
we consider the following cases. \\
\begin{inparaenum}
\item[(1)] If ${\rm Irrat}(\Sigma)=\{p_1,p_2\}$ with $p_1\neq p_2$,
then $\Gamma \equiv 6H - D$ and  $D_E=\alpha_1 E_1 + \alpha_2 E_2$
with $\alpha_i$ non--negative integers.
One has $-\kappa=2(\alpha_1+\alpha_2)$.  On the other hand $F\cdot \Gamma=12-(\alpha_1+\alpha_2)
\geqslant 2$ since $g>1$ thus, from \eqref{eq:DE}, $-\kappa\leqslant 20$.\\
\item[(2)] If ${\rm Irrat}(\Sigma)=\{p\}$,
then $\Gamma \equiv 6H - D$, with $D_E= \alpha
E_1+\beta_1F_1+\beta_2F_2$ for some non-negative integers
$\alpha, \beta_1,\beta_2$. One has $K_S\cdot F_1=K_S\cdot F_2=0$,
thus $-\kappa  =  2\alpha$. As before, $F\cdot \Gamma=12-\alpha
\geqslant 2$, hence again $-\kappa\leqslant 20$.
\end{inparaenum}

In any case, \eqref{eq:CZquartic-bis} holds.

\subsubsection{Ruled elliptic normal quartics: Case (c)}
This  case is treated in \cite[\S\;2.3.2, Case\;C2-2]{IsNa}.

Let $Y$ be a smooth, irreducible elliptic curve, and let $q \in
Y$.
Taking $0 \neq \xi \in {\rm Ext}^1(\cO_Y(q), \cO_Y(q))$ we
consider the corresponding rank two vector bundle $\cE:= \cE_{\xi}
$ on $Y$.
We let $X := \PP(\mathcal E)$
 and  $F_{q} := \sigma^*(q)$, where
$\sigma :X \to Y$ is the structure morphism. Let $H_{\cE}$ be the
tautological divisor class on $X$ and $F$ the numerical class of a
fibre.

On $X$ we consider a section $\Xi_0 \sim H_{\cE} - F_q$
corresponding to  $\cE\to\!\!\to \cO_Y(q)$, so that $\Xi_0^2 = 0$.
Notice that $|\Xi_0| = \{\Xi_0\}$, as it follows from $h^0(Y,
\cE(-q))=1$.  One has $|-K_X| = \{G\}$, where $G:= 2 \Xi_0$.

Since $|H_{\cE} - K_X| = |3 \Xi_0 + F_q|$, one finds that ${\rm
Bs}(|3 \Xi_0 + F_q|)$ consists of the single point $b := \Xi_0
\cap F_q$. The general member $B \in |3 \Xi_0 + F_q|$ is a smooth,
non--hyperelliptic curve of genus $3$ with $B \cdot \Xi_0 = 1$.
Thus $(X,B,G)$ is a $\mathcal{C}_2$-triplet. The separation
of $B$ and $G$ proceeds
 in two steps as follows.

On the first step, we let $\rho_1 : X_1 \to X$ be the blowup
at $b$ with exceptional divisor $\Delta_1$, and let $G_1 =
\rho_1^*(G) - \Delta_1$, $B_1= \rho_1^*(B) - \Delta_1$, $\Xi_0' =
\rho_1^*(\Xi_0) - \Delta_1$, and $F_q' = \rho_1^*(F_q) -
\Delta_1$. Since $B_1 \cdot G_1 =1$, we get a 
$\mathcal{C}_1$-triplet $(X_1,B_1,G_1)$ and a smooth point $b_1$
on $G_1$ with $\cO_{G_1}(B_1) \cong \cO_{G_1}(b_1)$. One has $b_1
\in \Delta_1$, because $G_1 = 2 \Xi_0' + \Delta_1$ and $B_1 \cdot
\Xi_0' = 0$, and furthermore, $b_1 \in {\rm Bs}(|B_1|)$ and  $b_1
\; |\!\!\!\!\!\in \; F'_q$.

On the second step,  we consider the blowup $\rho_2: S \to
X_1$  of $X_1$ at $b_1$ with exceptional divisor $\Delta_2$. In
this way, we arrive at a basic triplet $(S,H,E)$, where  $$H=B_1^*-\Delta_2= 3 \Xi_0'' + F_q'' +
3 \Delta_1' + 2 \Delta_2\,\,\,\mbox{and}\,\,\, E=G_1^*-\Delta_2= 2
\Xi_0'' + \Delta_1'\sim - K_S\,.$$ Here $\Xi_0''$, $F_q''$, and
$\Delta_1'$ are the proper transforms on $S$ of $\Xi'_0$, $F_q'$,
and $\Delta_1$, respectively. One has ${\Xi'}^ 2=-1, {\Delta'_1}^
2=-2$, and $\Xi'\cdot \Delta'_1=1$. Abusing notation, we still
denote by $F$ be total transform on $S$ of a ruling of $X$.

 The linear system $|H|$ is base point free of dimension 3,
with $H^ 2=4$ and $H\cdot F=3$. The associated morphism
$\pi=\phi_{|H|}$ sends $S$ to a quartic surface $\Sigma$ in
$\proj^3$ with a unique irrational singular point $p=\pi(E)$. It
is swept out by an elliptic pencil $|F|$ of rational normal
cubics. One has $E^ 2=-2$ and $p_g(p) = 2$.

For $C \sim 6H$
we have $\Gamma \equiv 6H - D$, with $D_E=\alpha \Xi'+\beta
\Delta'_1$ for some non-negative integers $\alpha$ and $\beta$.
Thus, from \eqref{eq:DE}, $-\kappa =  \alpha$. Since $\Gamma\cdot
F= 18-\alpha\geqslant 2$, then $-\kappa \leqslant 16$. Thus
again \eqref {eq:CZquartic-bis} holds.

\bigskip

This ends the discussion of the normal cases. We turn next to the
non--normal cases (see Proposition \ref
{prop:Umezu2}).

\subsubsection{Non--normal genus 2 scrolls}  In this case  $\Sigma$ is a cone
over an irreducible plane quartic $Y$ with a node or a cusp
(see \cite[Prop.\;2.6]{Ur}). The  one--dimensional singular locus
of $\Sigma$ is a double line $\ell$ passing through the vertex
$\mathfrak v$ of $\Sigma$.

The minimal desingularization of $\Sigma$ is the surface $S =
\PP(\cO_G \oplus L)$, where $G$ is the normalization of $Y$ and
 $L \in {\rm Pic}^4(G)$, $L\not\cong \omega^ {\otimes 2}_Y$. The
morphism $\pi : S \to \Sigma \subset \PP^3$ is determined by the
tautological line bundle $\cO_S(1)$. Letting $H$ and $E$,
be the sections corresponding to $\cO_G \oplus L
\to\!\!\!\to L$ and to $\cO_G \oplus L \to\!\!\!\to \cO_G$,
respectively, we obtain  $H^2 = - E^2 = 4$, $\cO_S(1) =
\cO_S(H)$, $H \equiv E + 4 F$, and $K_S \equiv - 2 E - 2 F \equiv
- 2 H + 6 F$, where $F$ stands, as usual, for the class of a
ruling of $S$.

Let $C \sim 6 H$ be an irreducible curve of geometric genus $g \in \{14,\;15\}$ on
$\Sigma$, and $\Gamma$ its proper transform on $S$.
Since $C$ is cut out on $\Sigma$ by a smooth surface of degree 6,
we have $\Gamma \equiv 6H - \alpha E$, where $\alpha\in\{0,1\}$.
Thus $-\kappa = -K_S \cdot \Gamma = 12+6\alpha\le 18$, and so
 \eqref {eq:CZquartic-bis} holds.
%

\medskip

Next we consider the non--normal elliptic scrolls, see Proposition
\ref {prop:Umezu2}.  There are two  types of such scrolls
described  in \cite[\S\;1]{Ur} as cases (II-1) and (II-2).

\subsubsection{Non--normal elliptic scrolls: case (a)} \label{ssec:nnorm(a)}
Take a smooth, irreducible elliptic curve $G$. Let $N \in {\rm
Pic}^0(G)$ be non-trivial, and let $M \in {\rm Pic}^2(G)$.
Consider the ruled surface  $S := \PP(\cO_G \oplus N)$ together
with the structure morphism $\sigma: S \to G$. Let $D_1$ and $D_2$
be the  sections associated
 with
$\cO_G \oplus N \to\!\!\to \cO_G$ and $\cO_G \oplus N \to\!\!\to
N$, respectively, and let $F$ be the ruling class.  Then $K_S
\equiv - 2 D_1$, and  the line bundle $H:= \cO_S(D_1) \otimes
\sigma^*(M)$ induces a finite birational morphism $\pi=\phi_{|H|}
: S \to \PP^3$ onto an irreducible quartic surface $\Sigma$.

The images $\ell_1= \pi(D_1)$ and $\ell_2= \pi(D_2)$ are skew
double lines of $\Sigma$, and ${\rm Sing}(\Sigma) = \ell_1 \cup
\ell_2$.   The image under $\pi$ of any fibre of $S$ is a line
meeting both $\ell_1$ and $\ell_2$. The general plane section of
$\Sigma$ has two nodes at the intersection points with $\ell_1$
and $\ell_2$,
 and its normalizations is $G$.

Let $C \sim 6 H$ be  an irreducible curve  of geometric genus $g$
on $\Sigma$. Then $\Gamma \sim 6 H\equiv 6 D_1 + 12 F$, and so,
$-\kappa = -K_S \cdot \Gamma = 24$,
proving again \eqref {eq:CZquartic-bis}.

\begin{rem} The construction of this scroll is classical.  Once  skew lines
$\ell_1, \ell_2$ in $\PP^ 3$  have been fixed, take a smooth,
irreducible, elliptic curve $G$  and two degree two maps $f_i:
G\to \ell_i$, $ i=1, 2$. For each $x\in G$, let $\ell_x=\langle
f_1(x), f_2(x)\rangle$. Then  $\Sigma=\cup_{x\in G}\ell_x$.
\end{rem}

\subsubsection{Non--normal elliptic scrolls: case (b)}  Let $G$
be  a smooth, irreducible elliptic curve, and let $M \in {\rm
Pic}^2(G)$ be a line bundle on $G$.  Let $\cE := \cE_{\xi}$ be the
rank--two vector bundle on $G$ fitting in the non-split sequence
$$0 \to \cO_G \to \cE \to \cO_G \to 0$$
associated to the choice of a  non--zero $\xi \in {\rm
Ext^1}(\cO_G, \cO_G)$. Let $S := \PP(\cE)$, together with the
structure morphism $\sigma: S \to G$ and the fibre class  $F$. Let
$D_1$ be the section corresponding to  $\cE \to\!\!\to \cO_G$, and
let $H := \cO_S(D_1) \otimes \sigma^*(M)$. As in \S~\ref
{ssec:nnorm(a)}, $H$ defines a finite morphism $\pi=\phi_{|H|} : S
\to \PP^3$  onto  an irreducible quartic surface $\Sigma$. Then
$\pi (D_1) $ is a double line $\ell$ of $\Sigma$, and any fibre of
$S$ is sent via $\pi$ to a line of $\Sigma$ crossing $\ell$. One
has ${\rm Sing}(\Sigma) =\ell$.  The general plane section $H$ of
$\Sigma$ has an $A_3$-singularity at $H\cap \ell$, and its
normalizations is $G$. Since $\Gamma \sim 6 H \equiv 6 D_1 + 12 F$ and $K_S \equiv - 2
D_1$, the computations go  as in \S~\ref {ssec:nnorm(a)}
proving \eqref {eq:CZquartic-bis}.

\smallskip

According to Proposition \ref {prop:Umezu2}, we are left with the
rational case. This gives rise to three items (see  (III-A),
(III-B), (III-C) in \cite[\S\;1]{Ur}).

\subsubsection{Rational non--normal quartics: case (a),
the Segre surface}  The {\em Segre surface} $\Sigma \subset \PP^3$
is the image of a normal surface $\widehat{\Sigma} \subset \PP^4$
of degree $4$ with at most Du Val singularities under the linear
projection $\Pi_{p}: \PP^4 \dasharrow \PP^3$ with  center
$p\not\in \widehat{\Sigma}$.  The surface $\widehat{\Sigma}$ is
the anticanonical image of a  weak
del Pezzo surface of degree $4$, i.e., the blowup of $\PP^2$ at
$5$ points,
see \cite {Seg}.

Apart from its one--dimensional singular locus $\Lambda$, which is
in general a double conic (see Remark in \cite[p. 277]{Ur} for
details), $\Sigma$ can have further isolated Du Val singularities
off $\Lambda$.
We have $-\kappa = -K_S \cdot \Gamma=  6 K_S^2 =  24$,
proving \eqref{eq:CZquartic-bis}.

\subsubsection{Rational non--normal quartics: case (b)}
In this case $\Sigma$ has a singular line $\ell$, and its general
plane section has geometric genus 2, hence it is either nodal or
cuspidal.  Besides,  $\Sigma$ may have isolated Du Val
singularities off $\ell$
(cf. \cite[Case\;(III-C),\;p. 269]{Ur}).

The surface $S$ is obtained  by successively blowing--up $\PP^2$
at $9$ points. The morphism $\pi:S \to \Sigma\subset  \PP^3$ is
defined by $H \sim 4L - 2 E_1 - \sum_{i=2}^9 E_i$, where $L$ is
the proper transform of a line in $\PP^2$ and $E_i$, for
$i=1,\ldots,9$,
are the (total) exceptional divisors of the blowups. In particular, $K_S
= - 3   L  + \sum_{i=1}^9 E_i$, and $h^ 0(S, \mathcal
O_S(-K_S))=1$, i.e., there is only one cubic curve  on $\proj^2$
 passing through the 9 blown--up points, corresponding to a
unique  effective  anticanonical divisor $D$ on $S$. Note that
$\pi(D)=\ell$ and that $\Sigma$ is swept out by a pencil of conics
cut out by the planes containing $\ell$, with the pullback
$\Lambda\sim L-E_1$  on $S$.

The surface $\Sigma$ can have further singularities along $\ell$
described in  \cite[Case\;(III-C),\;p. 269]{Ur}.  Consider the
morphisms $S \stackrel{\rho}{\to}\widehat{\Sigma}
\stackrel{\nu}{\to} \Sigma$, where $\nu$ is the normalization,
$\rho$ is  the minimal resolution of singularities of
$\widehat{\Sigma}$, and $\pi=\nu\circ \rho$.   Then any singular
point of $\widehat{\Sigma}$ which is not Du Val lies on $\nu^
{-1}(\ell)$ and  is a rational triple point. The number of such
triple points is at most $2$, their images on $\Sigma$ are also
triple points for $\Sigma$ (which in this case is a monoid). If
$\Delta$ is the fundamental cycle of such a triple point, then
$\Delta<D$. Moreover  $\Lambda\cdot \Delta=1$. Let $A$ be an
irreducible component of $\Delta$. Since $H=\Lambda-K_S$ and
$H\cdot A=0$, we have $K_S\cdot A=\Lambda\cdot A\in \{0,1\}$.
Since $A\cong \PP^ 1$, then $A^ 2\in \{-2,-3\}$.

For two possible rational triple points of $\Sigma$, we let
$\Delta_1$, $\Delta_2$ denote their fundamental cycles on $S$. If
$C \sim 6 H$ is an  irreducible curve of geometric genus $g$ on
$\Sigma$ with the proper transform $\Gamma$ on $S$, then $\Gamma
\equiv 6 H - \Delta'_1 -  \Delta'_2 $, where the support of
$\Delta'_i$ is contained in the support of $\Delta_i$, for $i=1,
2$. If $A_i$ is the unique component of $\Delta_i$ such that
$\Lambda\cdot A_i=1$, we let $\alpha_i$ be the multiplicity of
$A_i$ in $\Delta'_i$, for $ i=1, 2$. Thus, $-\kappa = -K_S \cdot
\Gamma = 12 + K_S\cdot (\Delta'_1 +
\Delta'_2)=12+(\alpha_1+\alpha_2)$. On the other hand, since
$\Gamma$ has genus $g\geqslant 14$, we have $2\leqslant
\Gamma\cdot \Lambda=12-(\alpha_1+\alpha_2)$. Hence
$-\kappa\leqslant 22$, and so,
%
\eqref{eq:CZquartic-bis} holds.

\subsubsection{Rational non--normal quartics: case (c)}
In this case $\Sigma$ is the image of a smooth surface $S \subset
\PP^5$ via the linear projection $\Pi_{\ell}: \PP^5 \dasharrow
\PP^3$ with center a line $\ell \subset \PP^5$ disjoint from $S$
(cf. \cite[Case (III-A)]{Ur}). In this situation, $\pi =
{\Pi_{\ell}}_{|S}\colon S\to\Sigma$ is the minimal
desingularization of $\Sigma$,  and
either:\\
\begin{inparaenum}
\item[(1)] $S$ is the \emph{Veronese surface},
i.e., $S\cong \PP^ 2$ embedded in $\PP^ 5$ via
the linear system $|2L|$, where $L$ is a line on $\proj^2$, and
$\Sigma$ is the \emph{Steiner's Roman surface}, or   \\
\item[(2)] $S \cong  \mathbb{F}_0$, and its embedding in $\PP^5$
is given by $|F_1 + 2 F_2|$, where $F_1$ and $F_2$ are the two distinct rulings, or\\
\item[(3)] $S = \mathbb{F}_2$, and its embedding in $\PP^5$
is given by $|D + F|$, where $D$ is a section with $D^ 2=2$, and
$F$ is the ruling.
\end{inparaenum}

Let  $C \sim 6 H$ on $\Sigma$ and $\Gamma$ on $S$ be as before.
Then in all cases $-\kappa = 36$, hence
\eqref{eq:CZquartic-bis}  does not hold. Therefore we have to
directly check if \eqref {eq:CZ1bis} holds. Since $\Phi (4,6)=73$,
\eqref  {eq:CZ1bis} holds if $\dim(\mathcal F)<23$. To see that
this is the case,  consider the \emph{Rohn exact sequence}
\[
0\to \mathcal O_S(1)^ {\oplus 2}\to N_{S|\PP^ 5}\to N_\pi\to 0
\]
(see \cite [(2.2)]{cil}), where $N_\pi$ is the normal sheaf to the
map $\pi$. Since $h^ 1(S,\mathcal O_S(1))=0$ in the above three
cases, one has
\[
\dim {\mathcal F}\leqslant h^ 0(S,N_\pi)=h^ 0(S,N_{S|\PP^ 5})-12.
\]
On the other hand, $h^ 0(S,N_{S|\PP^ 5})$ is the dimension of the
component of Hilbert scheme described by the surfaces $S$ in the
cases (1)--(3). Notice that the  surfaces in case (3) are
specializations of the ones in case (2). Finally, we obtain
\[
h^ 0(S,N_{S|\PP^ 5})=\dim (\Aut(\proj^5))-\dim (\Aut(S))=
 \left\{\begin{array}{cc}
27 & \mbox{in case} \; (1) \\
 & \\
29 & \mbox{in case} \; (2).\end{array} \right.
\]
(we leave the details to the reader; alternatively,
see the proof of \cite [Lemma (2.3)]{cil}). In conclusion, $\dim(\mathcal F)\leqslant 17<23$, which finishes
our proof.

\providecommand{\bysame}{\leavevmode\hboxto3em{\hrulefill}\thinspace}

\end{document}